\documentclass[a4paper]{amsart}
\usepackage[utf8]{inputenc}
\usepackage[T1]{fontenc}
\usepackage{lmodern}
\usepackage{microtype}
\usepackage{csquotes}
\usepackage{mathtools}
\usepackage[american]{babel}
\usepackage[
 url=false,
 backend=bibtex
]{biblatex}
\usepackage[
 colorlinks,
 pdfprintscaling=None,
]{hyperref}
\usepackage[
]{cleveref}
\bibliography{bibfile}
\theoremstyle{plain}
\newtheorem{theo}{Theorem}[section]
\newtheorem{lemm}[theo]{Lemma}
\newtheorem{coro}[theo]{Corollary}
\newtheorem{prop}[theo]{Proposition}
\theoremstyle{definition}
\newtheorem{defi}[theo]{Definition}
\newtheorem{prob}[theo]{Problem}
\theoremstyle{remark}
\newtheorem{rema}[theo]{Remark}
\newtheorem{exam}[theo]{Example}
\numberwithin{equation}{theo}

\newcommand{\N}{\mathbb{N}}

\newcommand{\R}{\mathbb{R}}
\newcommand{\id}{\mathrm{id}}
\newcommand{\bbE}{\mathbb{E}}
\newcommand{\Lin}{\mathcal{L}}
\DeclareMathOperator{\End}{\mathcal{L}}
\DeclarePairedDelimiter\paren{(}{)}
\DeclarePairedDelimiter\coint{[}{)}
\DeclarePairedDelimiter\set{\{}{\}}
\DeclarePairedDelimiter\abs{\lvert}{\rvert}
\DeclarePairedDelimiter\norm{\lVert}{\rVert}
\DeclarePairedDelimiterX\normx[2]{\lVert}{\rVert_{#2}}{#1}
\renewcommand{\bar}{\overline}
\renewcommand{\tilde}{\widetilde}
\renewcommand{\mid}{:}

\begin{document}
\title[Uniform stability of linear evolution equations]{Uniform stability of linear evolution equations, with applications to parallel transports}
\author{Tim Kirschner}
\address{Lehrstuhl für Komplexe Analysis\\ Universität Bayreuth}
\email{tim.kirschner@uni-bayreuth.de}
\urladdr{http://timkirschner.tumblr.com}
\date\today
\subjclass[2010]{Primary 34G10, 47D06; Secondary 53B05}
\begin{abstract}
I prove the bistability of linear evolution equations $x' = A(t)x$ in a Banach space $E$, where the operator-valued function $A$ is of the form $A(t) = f'(t)G(t,f(t))$ for a binary operator-valued function $G$ and a scalar function $f$. The constant that bounds the solutions of the equation is computed explicitly; it is independent of $f$, in a sense.

Two geometric applications of the stability result are presented. Firstly, I show that the parallel transport along a curve $\gamma$ in a manifold, with respect to some linear connection, is bounded in terms of the length of the projection of $\gamma$ to a manifold of one dimension lower. Secondly, I prove an extendability result for parallel sections in vector bundles, thereby answering a question by Antonio J. Di Scala.
\end{abstract}
\maketitle
\tableofcontents

\section{Introduction}
\label{s:intro}

Let $r$ be a natural number and $A \colon I \to \R^{r\times r}$ a, say continuous, function defined on an interval. Consider the linear evolution equation
\[
x' = A(t)x
\]
in $r$-dimensional space. A solution of the equation, or simply a solution for $A$, or better for $(A,r)$, is per definitionem a differentiable function $\phi \colon I' \to \R^r$ defined on an interval $I' \subset I$ such that, for all $\tau \in I'$, we have
\[
\phi'(\tau) = A(\tau)\phi(\tau).
\]

Recall \cite[112]{MR0352639} that (the equation associated to) $A$, or $(A,r)$, is called \emph{uniformly right stable} when there exists a real number $C > 0$ such that, for all solutions $\phi \colon I' \to \R^r$ and all $s,t \in I'$ with $s \leq t$, we have
\[
\norm{\phi(t)} \le C \norm{\phi(s)}.
\]
Here $\norm\cdot$ denotes the Euclidean norm on $\R^r$.
When $A$ is uniformly right stable, then, in particular, for all $s \in I$, all solutions $\phi$ defined on $I_{\geq s} = \{\tau \in I \mid \tau \ge s\}$ are bounded above in norm. Moreover, for all $t \in I$, all solutions $\psi$ defined on $I_{\le t}$ with $\psi(t) \neq 0$ are bounded away from zero (in norm).

Just as there is a notion of uniform right stability, there is a notion of \emph{uniform left stability}: You require the very last inequality above for all $s,t \in I'$ with $t \le s$ (instead of $s \le t$). The linear evolution equation given by $A$ is called \emph{bistable} when $A$ is both uniformly right stable and uniformly left stable; cf. \cite[113]{MR0352639}. The following rather elementary proposition (for a proof see, e.g., \cite[54]{MR1929104}) yields a first sufficient criterion for bistability.

\begin{prop}
\label{p:elementary_gronwall}
Let $r$ and $A$ be as above. Then for all solutions $\phi \colon I' \to \R^r$ for the equation associated to $A$ and all $s,t \in I'$, the inequality
\[
\norm{\phi(t)} \le e^{\abs*{\int_s^t \normx{A}{op} \,d\lambda}} \norm{\phi(s)}
\]
holds,
where $\normx\cdot{op}$ signifies the natural operator norm for $r\times r$ matrices.
\end{prop}

As a matter of fact, \cref{p:elementary_gronwall} tells us that when the norm of $A$ has finite integral over $I$ (either in the Lebesgue or the possibly improper Riemann sense), then $A$ is bistable. We can indeed take
\[
C = e^{\int_I \normx{A}{op} \,d\lambda}
\]
in the definition. This criterion is, however, by no means a necessary criterion. Consider the example $r=1$, $I = \coint{0,\infty}$, or $I = \R$, and $A$ given by $A(\tau) = \cos\tau$, viewed as a $1\times1$ matrix, for all $\tau \in I$. Then the integral of $\normx{A}{op} = \abs{\cos}$ over $I$ is evidently not finite. Yet, when $\phi \colon I' \to \R^1$ is a solution for $A$, a little elementary calculus proves the existence of an element $v \in \R^1$ such that
\[
\phi(\tau) = e^{\sin\tau}v, \quad \forall\tau \in I'.
\]
In consequence, we have
\[
\norm{\phi(t)} \le e\norm v = e^2e^{-1}\norm v \le e^2 \norm{\phi(s)}
\]
for all $s,t \in I'$; that is, we have bistability for $A$.

The latter example generalizes as follows. Let $r$ and $A$ be again arbitrary. Assume that, for some $t_0 \in I$, the function
\[
\Phi \colon I \to \R^{r\times r}, \quad \Phi(\tau) = \exp\paren*{\int_{t_0}^\tau A \,d\lambda},
\]
where $\exp$ denotes the matrix exponential function for $r\times r$ matrices, solves the equation associated to $A$ in the sense that $\Phi$ is differentiable with
\[
\Phi'(\tau) = A(\tau)\Phi(\tau), \quad \forall \tau \in I.
\]
Then every solution $\phi\colon I' \to \R^r$ for $A$ can be written as $\Phi|_{I'} v$ for some element $v \in \R^r$. In consequence, we have
\begin{align*}
\norm{\phi(t)} = \norm{\Phi(t)v} &\le \normx{\Phi(t)}{op}\norm{v} = \normx{\Phi(t)}{op}\norm{\Phi(s)^{-1}\Phi(s)v} \\ &\le \normx{\Phi(t)}{op}\normx{\Phi(s)^{-1}}{op}\norm{\Phi(s)v} = \normx{\Phi(t)}{op}\normx{\Phi(s)^{-1}}{op}\norm{\phi(s)}
\end{align*}
for all $s,t \in I'$. Furthermore, for all $\tau \in I$, we have
\[
\normx{\Phi(\tau)^{\pm1}}{op} \le e^{\normx*{\pm\int_{t_0}^\tau A \,d\lambda}{op}}.
\]
Thus when
\[
M := \sup\set*{\normx*{\int_{t_0}^\tau A \,d\lambda}{op} \mid \tau \in I} < \infty,
\]
we infer that $A$ is bistable. Indeed we can take $C = e^{2M}$ in the definition.

Observe that for general $r$ and $A$, an element $t_0 \in I$ such that $\Phi$, defined as above, solves the equation associated to $A$ does not exist. The typical examples where such a $t_0$ exists are subsumed under the name of \emph{Lappo-Danilevskii}; cf. \cite[86]{MR1351004}. I refrain from elaborating on this point. Instead I ask whether the criterion we have established for Lappo-Danilevskii systems extends to arbitrary systems.

\begin{prob}
\label{q:criterion}
Let $r$ be a natural number, $I$ an interval, $A \colon I \to \R^{r\times r}$ continuous. Assume that there exists a number $M > 0$ such that
\[
\norm*{\int_s^t A \,d\lambda} \leq M, \quad \forall s,t \in I.
\]
Is it true then that $A$ is bistable?
\end{prob}

For the moment I feel that \cref{q:criterion} is out of reach---at least, if one does not make additional structural assumptions on $A$ (e.g., periodicity, upper triangular form, Lappo-Danilevskii form, or eigenvalue estimates). Luckily, for the purposes of this note, we are in the position to make additional structural assumptions on $A$. Specifically, I investigate the stability of $A$ in case
\[
A(\tau) = f'(\tau)\tilde G(\tau,f(\tau)), \quad \forall\tau \in I,
\]
for a function $\tilde G \colon I\times J \to \R^{r\times r}$ and a, say $C^1$, function $f \colon I \to J$, where $J \subset \R$ is an interval. My stability results are presented and proven in \cref{s:stability}. The main theorem is \cref{t:main}. In view of \cref{q:criterion}, or more generally from the point of view of a stability theorist, the striking feature of \cref{t:main} is that it proves bistability for a wide range of systems $A$ for which all of the customary criteria fail---in particular, the integral $\int_I \normx{A}{op} \,d\lambda$ is not finite, $A$ is not periodic, and $A$ is not Lappo-Danilevskii (see \cref{x:elementary}).

From the point of view of differential geometry, systems $A$ of the described form are interesting as they occur in the study of parallel transports with respect to a linear connection on a vector bundle. In fact, effective stability results correspond to effective bounds for parallel transports. This application is explicated in \cref{s:bounds}. A further geometric application is given in \cref{s:neg}. The latter deals with the possibility to extend parallel sections in a vector bundle which are defined on the complement of the graph of a continuous function.
\medskip

\emph{Acknowledgements}: I simply must thank Antonio J. Di Scala, for he keeps bringing beautiful mathematics to people (like me).

\section{A change of variables formula for the Bochner integral}
\label{s:cov}

For lack of an adequate reference in the literature I state and prove here a change of variables formula (“integration by substitution”) for the Bochner integral.

\begin{theo}
\label{t:cov}
Let $I,J \subset \R$ be intervals, $E$ a Banach space, $y \in L^\infty_{loc}(J;E)$, and $f \in AC_{loc}(I)$ such that $f(I) \subset J$. Then $f' (y \circ f) \in L^1_{loc}(I;E)$ and, for all $s,t \in I$, we have
\begin{equation}
\label{e:cov}
\int_s^t f' (y \circ f) \,d\lambda = \int_{f(s)}^{f(t)} y \,d\lambda.
\end{equation}
\end{theo}

Here go my conventions concerning notation and terminology. An \emph{interval} is a connected, or equivalently a convex, subset of the real number line. The empty set is an interval. A \emph{Banach space} is a real Banach space, just as a \emph{vector space} without further specification is a real vector space. When $F$ is a normed vector space, I denote $\normx\cdot F$ the norm of $F$. When I feel that $F$ can be guessed from the context, I might drop the reference to it, thus writing $\norm\cdot$ instead of $\normx\cdot F$.

Let $I$ be an interval, $E$ a Banach space. Then $L^1(I;E)$ denotes the set of all \emph{Bochner integrable} functions $x \colon I \to E$; see \cite[9]{MR1886588}. Functions that agree almost everywhere on $I$ are not identified (somewhat contrary to etiquette). $L^\infty(I;E)$ denotes the set of all \emph{Bochner measurable} functions $x \colon I \to E$ which are \emph{essentially bounded}---that is, there exists a real number $M$ such that the set
\[
\set{t \in I \mid \normx{x(t)}E > M}
\]
is a (Lebesgue) null set. $C(I;E)$ denotes the set of all \emph{continuous} functions from $I$ to $E$. $AC(I;E)$ denotes the set of all \emph{absolutely continuous} functions $x \colon I \to E$; see \cites[16]{MR1886588}[73]{MR2527916}. I put $C(I) := C(I;\R)$ and $AC(I) := AC(I;\R)$, where $\R$ is thought of as a normed vector space. $L^1_{loc}(I;E)$ denotes the set of all functions $x \colon I \to E$ such that, for all compact intervals $K \subset I$, we have $x|_K \in L^1(K;E)$. Similar definitions apply for $L^\infty_{loc}$ and $AC_{loc}$. When $f \in AC_{loc}(I)$, I denote $f'$ the derivative of $f$, which is defined to be the ordinary derivative of $f$ at points where $f$ is differentiable and $0$ otherwise.

When $x \in L^1(I;E)$, I write $\int_I x \,d\lambda$ or $\int x \,d\lambda_I$ for the \emph{Bochner integral} of $x$ on $I$ in $E$. Note that the reference to $E$ is suppressed in the notation of the integral (actually a bad thing, but I surrender to the customs here). The $\lambda$ shall hint at an integration with respect to the \emph{Lebesgue measure}, or better its trace on $I$. Yet, as I will not integrate with respect to measures other than the Lebesgue measure, I may abstain from the general formalism. When $x \in L^1_{loc}(I;E)$ and $s,t \in I$, I set
\[
\int_s^t x \,d\lambda := \begin{cases}\int x|_K \,d\lambda_K & \text{when } s\le t,\\ -\int x|_K \,d\lambda_K & \text{when } t < s, \end{cases}
\]
where $K$ signifies the set of all points lying between $s$ and $t$---that is,
\[
K := \set{\alpha s + (1-\alpha)t \mid \alpha \in [0,1]},
\]
which is a subset of $I$.

\begin{rema}
\label{r:cov}
\Cref{t:cov} should be seen as preliminary, or exemplary, in the following respect. Assume that $E = \R$ and weaken the conditions $y \in L^\infty_{loc}(J;E)$ and $f \in AC_{loc}(I)$ respectively to $y \in L^1_{loc}(J;E)$ and $f \colon I \to \R$ being a function which is differentiable almost everywhere on $I$. Then the conclusion of the theorem---that is, $f'(y\circ f) \in L^1_{loc}(I;E)$ and, for all $s,t \in I$, we have \cref{e:cov}---holds if and only if $Y \circ f \in AC_{loc}(I;E) = AC_{loc}(I)$ for one, or equivalently all, primitive(s) $Y \colon J \to E$ of $y$. This fact is due to Serrin and Varberg \cite[Theorem 3]{MR0247011}; see also \cite[Theorem 3.54]{MR2527916}. The variant stated in \cref{t:cov} is implied as a special case; cf. \cites[Corollary 7]{MR0247011}[Corollary 3.59]{MR2527916}.

The theorem of Serrin and Varberg extends to finite-dimensional spaces $E$ of course. For arbitrary Banach spaces $E$, however, the argument of Serrin and Varberg breaks down as it uses, among others, the fact that $Y\circ f \in AC_{loc}(I;E)$ implies the almost everywhere differentiability of $Y\circ f$ on $I$. Yet even presupposing that $Y\circ f$ be differentiable almost everywhere on $I$ in the if-part, Serrin's and Varberg's argument cannot be copied naively.\footnote{I suggest that some research in that direction be carried out.} For my proof of \cref{t:cov} I managed to retain an essential portion of Serrin's and Varberg's ideas---namely, \cref{l:serrin-varberg}.
\end{rema}

\begin{lemm}
\label{l:serrin-varberg}
Let $I \subset \R$ be an interval, $f \colon I \to \R$ a function, $A \subset I$ such that $f$ is differentiable at every point of $A$. Assume that $f(A)$ is a null set. Then the set $\set{t \in A \mid f'(t) \neq 0}$ is a null set.
\end{lemm}

\begin{proof}
See the original \cite[515]{MR0247011}, or \cite[95]{MR2527916} where the former source is copied.
\end{proof}

\begin{proof}[Proof of \cref{t:cov}]
To begin with, remark that it suffices to treat the case where $I$ and $J$ are compact (the general case then follows readily). So, assume $I$ and $J$ compact. When $I$ is the empty set, the assertion is clear. So, assume that $I$ is inhabited. Then $J$ too is inhabited (since $f$ is a function from $I$ to $J$). Thus there exist unique real numbers $c\le d$ such that $J = [c,d]$. In turn, it makes sense to define
\[
Y \colon J \to E, \quad Y(u) = \int_c^u y \,d\lambda.
\]
Since $y \in L^\infty(J;E)$, we deduce that $Y \colon J \to E$ is Lipschitz continuous. Hence $f \in AC(I)$ implies that $Y \circ f \in AC(I;E)$.\footnote{I omit the details yielding this and the previous assertion.}

Denote $J_0$ the set of all elements $u$ of $J$ such that $Y \colon J \to E$ is differentiable in $u$ and  $Y'(u) = y(u)$. Denote $I_0$ the set of all $t \in I$ such that $f$ is differentiable at $t$ and we have $f(t) \in J_0$ if $f'(t) \neq 0$. Define $A$ to be the set of all $t \in I$ such that $f$ is differentiable at $t$ and $f(t) \notin J_0$. Then $f(A) \subset J \setminus J_0$. Thus $f(A)$ is a null set since $J \setminus J_0$ is a null set; see \cite[Proposition 1.2.2 a)]{MR1886588}. Therefore $\{t \in A \mid f'(t) \neq 0\}$ is a null set by means of \cref{l:serrin-varberg}. Since $f \in AC(I)$, the set of points of $I$ at which $f$ is not differentiable is a null set \cite[Proposition 3.8]{MR2527916}. Taking into account that
\[
I \setminus I_0 = \{t \in I \mid f \text{ not differentiable at } t\} \cup \{t \in A \mid f'(t) \neq 0\},
\]
we infer that $I \setminus I_0$ is a null set.

Now let $t \in I_0$. When $f'(t) = 0$, then, using the Lipschitz continuity of $Y$, one deduces that $Y\circ f$ is differentiable at $t$ with $(Y\circ f)'(t) = 0$.\footnote{I omit the details.} In particular,
\[
(Y \circ f)'(t) = f'(t)y(f(t)).
\]
When $f'(t) \neq 0$ on the other hand, we know that $f(t) \in J_0$, whence the differentiability of $Y\circ f$ at $t$ as well as the validity of latter formula follow from the traditional chain rule \cite[337]{MR1216137}.

The previous arguments show in particular that the set of elements of $I$ at which $Y\circ f$ is not differentiable is a null set (as it is contained in $I \setminus I_0$). Define $(Y\circ f)' \colon I \to E$ to be the function which is given by the derivative of $Y\circ f$ at points of differentiability of $Y\circ f$ and by $0 \in E$ otherwise.
Then according to \cite[Proposition 1.2.3]{MR1886588} we have $(Y \circ f)' \in L^1(I;E)$ and, for all $s,t \in I$,
\[
\int_s^t (Y\circ f)' \,d\lambda = (Y\circ f)(t) - (Y\circ f)(s).
\]
Furthermore, as the set where the functions $(Y\circ f)'$ and $f'(y\circ f)$ differ is a null set, we infer $f'(y\circ f) \in L^1(I;E)$ as well as
\[
\int_s^t f'(y\circ f) \,d\lambda = \int_s^t (Y\circ f)' \,d\lambda
\]
for all $s,t\in I$. Observing that, for all $s,t \in I$, we have
\[
(Y\circ f)(t) - (Y\circ f)(s) = Y(f(t)) - Y(f(s)) = \int_{f(s)}^{f(t)} y \,d\lambda,
\]
the intended \cref{e:cov} follows.
\end{proof}

We apply the change of variables formula to the context of evolution equations. The missing terminology is explained in \cref{d:sol}.
When $E$ is a Banach space, we write $\End(E)$ for the Banach space of continuous linear maps from $E$ to itself; cf. \cite[5--7]{MR1666820}. 

\begin{coro}
\label{c:subst}
Let $I$, $J$, $E$, and $f$ be as in \cref{t:cov}, $B \in L^\infty_{loc}(J;\End(E))$, and $A = f'(B \circ f)$. Then:
\begin{enumerate}
\item \label{i:subst-l1} $A \in L^1_{loc}(I;\End(E))$.
\item \label{i:subst-sol} When $y$ (resp. $Y$) is a vector (resp. operator) solution for $B$ in $E$, then the restriction of $x := y \circ f$ (resp. $X := Y \circ f$) to any subinterval of its domain of definition is a vector (resp. operator) solution for $A$ in $E$.
\item \label{i:subst-evolop} When $Y$ is an evolution operator for $B$ in $E$, then the composition $Y \circ (f\times f)$ is an evolution operator for $A$ in $E$. Here $f\times f$ signifies the function given by the assignment $(t,s) \mapsto (f(t),f(s))$.
\end{enumerate}
\end{coro}

\begin{proof}
\Cref{i:subst-l1} is immediate from \cref{t:cov} (applied to $\End(E)$ in place of $E$ and $B$ in place of $y$).

As to \cref{i:subst-sol}, let $y$ be a vector solution for $B$ with domain of definition equal to $J'$. Then $x$ is defined on $f^{-1}(J')$. Let $I'$ be a subinterval of $f^{-1}(J')$---that is, an interval such that $f(I') \subset J'$. Let $s,t \in I'$. Then again by \cref{t:cov} (this time applied to $I'$, $J'$, $f|_{I'}$, and $By|_{J'}$ instead of $I$, $J$, $f$, and $y$, respectively), we have
\begin{align*}
x(t) - x(s) & = y(f(t)) - y(f(s)) \\
& = \int_{f(s)}^{f(t)} By \,d\lambda
= \int_s^t f'((By) \circ f) \,d\lambda
= \int_s^t Ax \,d\lambda.
\end{align*}
Therefore $x|_{I'}$ is a vector solution for $A$. For operator solutions instead of vector solutions one argues analogously. \Cref{i:subst-evolop} is a direct consequence of the operator solution part of \cref{i:subst-sol}.
\end{proof}

\section{Stability theorems}
\label{s:stability}

Let me recall the fundamental theorem on the existence and the uniqueness of solutions of linear evolution equations in Banach spaces.

\begin{defi}
\label{d:sol}
Let $I \subset \R$ be an interval, $E$ a Banach space, $A \in L^1_{loc}(I;\End(E))$.
\begin{enumerate}
\item\label{i:sol-vector} A \emph{(vector) solution} for $A$ in $E$ is a continuous function $x \colon I' \to E$ defined on an interval $I' \subset I$ such that
\[
x(t) - x(s) = \int_s^t Ax \,d\lambda
\]
holds for all $s,t \in I'$, where $Ax \colon I' \to E$ is given by $(Ax)(\tau) = A(\tau)(x(\tau))$ for all $\tau \in I'$.
\item\label{i:sol-op} An \emph{operator solution} for $A$ in $E$ is a continuous function $X \colon I' \to \End(E)$ defined on an interval $I' \subset I$ such that
\[
X(t) - X(s) = \int_s^t AX \,d\lambda
\]
holds for all $s,t \in I'$, where $AX \colon I' \to \End(E)$ is given by $(AX)(\tau) = A(\tau) \circ X(\tau)$ for $\tau \in I'$.
\item\label{i:sol-evolop}
An \emph{evolution operator} for $A$ in $E$ is a function
\[
X \colon I \times I \to \End(E)
\]
such that, for all $s \in I$,
\begin{enumerate}
\item the function $X(\_,s)$ is an operator solution for $A$ in $E$, and
\item $X(s,s) = \mathrm{id}_E$.
\end{enumerate}
\end{enumerate}
\end{defi}

\begin{rema}
In \cref{i:sol-vector} of \cref{d:sol} the assumption that $Ax$ be an element of $L^1_{loc}(I';E)$ is implicit, just as in \cref{i:sol-op} the assumption that $AX$ be an element of $L^1_{loc}(I';\End(E))$ is implicit. Otherwise the integrals would not even make sense. 

These conditions are, however, automatic assuming the continuity of $x$ (resp. $X$). The proof sketch is this. Let $K \subset I'$ be a compact interval. Then as $x|_K \colon K \to E$ is continuous, it is Bochner measurable \cite[Corollary 1.1.2 c)]{MR1886588}. Moreover, $x|_K$ is bounded. Thus according to  \cite[VI, Corollary 5.12]{MR1216137}, $(Ax)|_K$ is an element of $L^1(K;E)$ as $A|_K \in L^1(K;\End(E))$ and the pairing
\[
\End(E) \times E \to E, \quad (B,y) \mapsto B(y)
\]
is bilinear and continuous. For $X$ in place of $x$ you use the pairing given by the composition of operators.
\end{rema}

\begin{theo}
\label{t:evolop}
Let $I$, $E$, and $A$ be as in \cref{d:sol}.
\begin{enumerate}
\item \label{i:evolop-ex} There exists a unique evolution operator for $A$ in $E$.
\item \label{i:evolop-trans} When $X$ is an evolution operator for $A$ in $E$, then, for all $s,t \in I$, the operator $X(t,s)$ is invertible and satisfies
\[
X(t,s)^{-1} = X(s,t).
\]
Moreover, for all $s,t,u \in I$, we have
\[
X(u,t)X(t,s) = X(u,s).
\]
\end{enumerate}
\end{theo}

\begin{proof}
See \cite[96--101]{MR0352639}.
\end{proof}

Our main tool for bounding the solutions of evolution equations is the following. 

\begin{lemm}
\label{l:compare}
Let $I \subset \R$ be an interval, $E$ a Banach space, $A_k \in L^1_{loc}(I;\End(E))$, $X_k$ an evolution operator for $A_k$ in $E$ ($k=1,2$), $N>0$, $\nu_1 \in \R$, and $\epsilon \in \{\pm1\}$ such that, for all $s,t \in I$ with $s \leq t$, we have
\[
\norm{X_1(t,s)^\epsilon} \leq Ne^{-\nu_1(t-s)}.
\]
Then, for all $s,t \in I$ with $s \leq t$, the following estimates hold:
\begin{align*}
\norm{X_2(t,s)^\epsilon} & \leq N e^{-\nu_1(t-s)} e^{N\int_s^t\norm{A_2 - A_1}\,d\lambda}, \\
\norm{X_2(t,s)^\epsilon - X_1(t,s)^\epsilon} & \leq N e^{-\nu_1(t-s)} \paren*{e^{N\int_s^t \norm{A_2 - A_1} \,d\lambda} - 1}.
\end{align*}
\end{lemm}

\begin{proof}
See \cite[III, Lemma 2.3]{MR0352639}.
%When $I$ is compact, this is . The general case follows immediately: given $s,t \in I$, $s \leq t$, simply apply the result to $A_k|_{[s,t]}$.
%
%Let us remark two things on the proof in \cite{MR0352639}. First of all, observe that \cite{MR0352639} treats only the “$+$” case. The “$-$” case, however, is obtained along the very same lines noting that $X_k(\_,s)^{-1}$, for given $s \in I$ and $k=1,2$, is an operator solution of the adjoint equation associated to $A_k$; that is, we have
%\[
%X_k(t,s)^{-1} - X_k(t_0,s)^{-1} = \int_{t_0}^t X_k(\_,s)^{-1}(-A_k) \,d\lambda
%\]
%for all $t,t_0 \in I$.
%
%Second of all, the proof in \cite{MR0352639} relies on a Gronwall inequality type result which, in the stated form, requires the function $p = \norm{A_2 - A_1}$ to be continuous. However, this continuity assumption can be weakened to local integrability as one infers, for instance, approximating $p$ (on some compact subinterval $I' \subset I$) by continuous (nonnegative) functions.
\end{proof}

\begin{coro}
\label{c:standard_estimate}
Let $I$, $E$, $A$ be as in \cref{d:sol}, $X$ an evolution operator for $A$ in $E$. Then, for all $s,t \in I$ with $s \leq t$, we have
\[
\norm{X(t,s)^{\pm1}} \leq e^{\int_s^t \norm{A} \,d\lambda}.
\]
\end{coro}

\begin{proof}
Take $A_1$ on $I$ to be constantly equal to the zero operator on $E$, $A_2 = A$, $X_1$ on $I\times I$ constantly equal to $\id_E$, $X_2 = X$, $N = 1$, and $\nu_1 = 0$ in \cref{l:compare}.
\end{proof}

Let $I \subset \R$ be an arbitrary subset, $n \in \N$. Then we define
\[
\mathcal S_n(I) := \{a = (a_0,\dots,a_n) \in I^{n+1} \mid (\forall i<n)\, a_i < a_{i+1}\}
\]
to be the set of all length-$(n+1)$, strictly increasing sequences of elements of $I$.
When $J$ is an interval and $F$ a Banach space, we make the set $L^1(J;F)$ (see \cref{s:cov}) into a normed vector space the usual way so that
\[
\normx{y}{L^1(J;F)} = \int_J \normx{y}F \,d\lambda
\]
for all $y \in L^1(J;F)$.

\begin{lemm}
\label{l:stability}
Let $J \subset \R$ be an interval, $n \in \N$, $a \in \mathcal S_n(\R)$, $E$ a Banach space, $F = \End(E)$, and $G = (G_0,\dots,G_{n-1})$ an $n$-tuple of elements of $L^\infty_{loc}(J;F) \cap L^1(J;F)$. Set
\[
N := e^{\max_{i < n} \normx{G_i}{L^1(J;F)}},
\]
where the maximum is taken to be $0$ in case $n = 0$,
\[
V := \sum_{0 < j < n} \normx{G_j - G_{j-1}}{L^1(J;F)},
\]
and $I := [a_0,a_n]$.
Then, for all $f \in AC(I)$ such that $f(I) \subset J$ and all functions $A \colon I \to F$ such that
\[
A(t) = f'(t)G_i(f(t))
\]
whenever $i<n$ and $t \in \coint{a_i,a_{i+1}}$,\footnote{Note that the value of the function $A$ at $a_n$ is arbitrary.} we have
\begin{enumerate}
\item \label{i:stability-l1} $A \in L^1(I;F)$, and
\item \label{i:stability-Xbound} when $X$ is an evolution operator for $A$ in $E$, then
\[
\norm{X(t,s)^{\pm1}} \leq N^2 e^{N^{3 + 2N} V}
\]
for all $s,t \in I$ with $s \leq t$.
\end{enumerate}
\end{lemm}

\begin{proof}
Let $f$ and $A$ be as above. Then, for all $i < n$, by \cref{i:subst-l1} of \cref{c:subst}, or directly by \cref{t:cov}, we have
\[
f'(G_i\circ f) \in L^1_{loc}(I;F) = L^1(I;F).
\]
Hence, for all $i<n$,
\[
\mathbf1_{\coint{a_i,a_{i+1}}}f'(G_i\circ f) \in L^1(I;F),
\]
but
\[
\mathbf1_{\coint{a_i,a_{i+1}}}f'(G_i\circ f) = \mathbf1_{\coint{a_i,a_{i+1}}}A.
\]
Thus $A \in L^1(I;F)$; that is, \cref{i:stability-l1} holds.

Now let $X$ be an evolution operator for $A$ in $E$. Observe that, for all $i < n$, since $G_i \in L^1(J;F)$, there exists a (unique) evolution operator $Y_i$ for $G_i$ in $E$ by \cref{i:evolop-ex} of \cref{t:evolop}. Put
\[
Y_i(v) := Y_i(v,c), \quad c := \inf f(I)
\]
for sake of brevity ($i < n$, $v \in J$), which makes sense since $f$ is continuous and $I$ is compact, so that $c \in f(I) \subset J$.

Let $s,t \in I$. In case $s,t \in [a_i,a_{i+1}]$ for some $i < n$, we have
\[
X(t,s) = Y_i(f(t),f(s))
\]
by means of \cref{c:subst}, \cref{i:subst-evolop} and the uniqueness of the evolution operator for $A|_{[a_i,a_{i+1}]}$ in $E$, which is due to \cref{i:evolop-ex} of \cref{t:evolop}. Moreover,
\[
Y_i(f(t),f(s)) = Y_i(f(t),c)Y_i(c,f(s)) = Y_i(f(t))Y_i(f(s))^{-1}
\]
according to \cref{i:evolop-trans} of \cref{t:evolop}. Assume $s \leq t$ now. Then evidently there exist natural numbers $k \leq l < n$ such that $s \in [a_k,a_{k+1}]$ and $t \in [a_l,a_{l+1}]$. Thus we may write
\begin{align*}
X(t,s) & = X(t,a_l) X(a_l,a_{l-1}) \dots X(a_{k+2},a_{k+1}) X(a_{k+1},s) \\
& = Y_l(f(t))Y_l(f_l)^{-1}Y_{l-1}(f_l) \dots Y_{k+1}(f_{k+1})^{-1}Y_k(f_{k+1})Y_k(f(s))^{-1} \\
& = Y_l(f(t)) I_l \dots I_{k+1} Y_k(f(s))^{-1},
\end{align*}
where we put
\[
I_j := Y_j(f_j)^{-1}Y_{j-1}(f_j), \quad f_j := f(a_j)
\]
for $0 < j < n$.

Note that for all $i < n$, by \cref{c:standard_estimate}, we have
\[
\norm{Y_i(v,u)^{\pm1}} \leq e^{\int_u^v \norm{G_i} \,d\lambda} \leq e^{\normx{G_i}{L^1(J;F)}} \leq N
\]
for all $u,v \in J$ such that $u \leq v$. Therefore for all $0 < j < n$, applying \cref{l:compare}, we deduce
\begin{align*}
\norm{Y_j(f_j)^{\pm1} - Y_{j-1}(f_j)^{\pm1}}
& \leq N \paren*{e^{N\int_c^{f_j} \norm{G_j - G_{j-1}} \,d\lambda} - 1} \\
& \leq N \paren*{e^{N \normx{G_j - G_{j-1}}{L^1(J;F)}} - 1} \\
& \leq N e^{N \normx{G_j - G_{j-1}}{L^1(J;F)}} N \normx{G_j - G_{j-1}}{L^1(J;F)} \\
& \leq N^2 N^{2N} \normx{G_j - G_{j-1}}{L^1(J;F)},
\end{align*}
where we employ the elementary estimate (look at the power series expansions on both sides)
\[
e^x - 1 \leq e^xx, \quad \forall x \geq 0,
\]
in order to obtain the second last inequality; the very last inequality follows using the triangle inequality
\[
\normx{G_j - G_{j-1}}{L^1(J;F)} \leq \normx{G_j}{L^1(J;F)} + \normx{G_{j-1}}{L^1(J;F)} \le 2\log N
\]
in the exponent.
In consequence, for all $0 < j < n$, as
\[
I_j - \id_E = (Y_j(f_j)^{-1} - Y_{j-1}(f_j)^{-1})Y_{j-1}(f_j),
\]
we infer
\begin{equation} \label{e:Ibound}
\begin{split}
\norm{I_j}
& \leq  \norm{I_j - \id_E} + \norm{\id_E} \\
& \leq  \norm{Y_j(f_j)^{-1} - Y_{j-1}(f_j)^{-1}}\norm{Y_{j-1}(f_j)} + 1 \\
& \leq N^2 N^{2N} \normx{G_j - G_{j-1}}{L^1(J;F)} N + 1 \\
& \leq e^{N^{3 + 2N} \normx{G_j - G_{j-1}}{L^1(J;F)}}
\end{split}
\end{equation}
eventually using the elementary estimate
\[
x + 1 \leq e^x, \quad \forall x \geq 0.
\]

Hence,
\[
\norm{I_l \dots I_{k+1}} \leq \prod_{j=k+1}^l \norm{I_j} \leq e^{N^{3 + 2N} \sum_{j = k+1}^l \normx{G_j - G_{j-1}}{L^1(J;F)}} \leq e^{N^{3 + 2N} V}.
\]
So finally we obtain
\[
\norm{X(t,s)} \leq \norm{Y_l(f(t))} \norm{I_l \dots I_{k+1}} \norm{Y_k(f(s))^{-1}} \leq N^2 e^{N^{3 + 2N} V}.
\]
This means we have proven the “$+$” case of \cref{i:stability-Xbound}. The “$-$” case is treated along the same lines. I give only an indication.
Remarking that, for all $0 < j < n$,
\[
I_j^{-1} - \id_E = Y_{j-1}(f_j)^{-1}(Y_j(f_j) - Y_{j-1}(f_j)),
\]
you adapt \cref{e:Ibound} in order to bound $\norm{I_j^{-1}}$. Then writing
\[
X(t,s)^{-1} = Y_k(f(s))I_{k+1}^{-1} \dots I_l^{-1}Y_l(f(t))^{-1},
\]
you finish up as before.
\end{proof}

\begin{rema}
\label{r:stability}
For a specific $f$ the upper bound in \cref{i:stability-Xbound} of \cref{l:stability} can be improved, possibly, by passing from $J$ to $f(I)$ and from $G_i$ to $G_i|_{f(I)}$ for all $i<n$. The new upper bound will then, however, depend on $f$ by way of depending on $f(I)$. In particular, the new upper bound is no longer uniform in $f$. For precisely this reason I have chosen not to use the latter idea in the formulation of \cref{l:stability}. The way the lemma stands, it stresses the fact that for given $J$, $E$, $G$ you have a single bound for all systems $A$---no matter what $f$ is, and(!) no matter what $a$ is.
\end{rema}

Let $F$ be a metric space. Then for any set $J$, the set $F^J$ of functions from $J$ to $F$ comes equipped with a \emph{supremum distance}, which I write $d_\infty$ (sloppily neglecting the references to $J$ and $F$). In concrete terms, given $x,y \in F^J$, we have
\[
d_\infty(x,y) = \sup\set{d_F(x(t),y(t)) \mid t \in J},
\]
where $d_F$ signifies the distance of the metric space $F$ and the supremum is taken with respect to the extended nonnegative reals $[0,\infty]$. Note that $d_\infty$ is generally not a metric on the set $F^J$ in the ordinary sense of the word as $d_\infty$ might take the value $\infty$. Nevertheless $d_\infty$ is what is called an \emph{extended (real-valued) metric}. A topology on $F^J$ is defined by $d_\infty$ the standard way, and continuity with respect to this topology has an $\epsilon$-$\delta$ characterization.

Now let $G \colon I \to F$ be a function such that $I \subset \R$. Then the \emph{total variation} of $G$ with respect to $F$ is defined as
\[
V = V_F(G) = \sup \left\{\sum_{i<n}d_F(G(a_i),G(a_{i+1})) \mid n \in \N,\, a \in \mathcal S_n(I)\right\},
\]
where the supremum is taken with respect to the extended nonnegative reals $[0,\infty]$ again. Specifically, we have $V = 0$ if $I = \emptyset$. We say that $G$ is of \emph{bounded variation} with respect to $F$ when $V < \infty$.

When $F$ is a normed vector space, the previous definitions apply passing from $F$ to its associated metric space. The latter is given by $d_F(x,y) = \normx{y-x}F$ for all $x,y \in F$ of course.

Let $n \in \N$ and $a \in \mathcal S_n(\R)$. Then the \emph{mesh} of $a$ is defined as the number
\[
\mu(a) := \max_{i<n} \abs{a_{i+1} - a_i} = \max_{i<n} (a_{i+1} - a_i),
\]
where the maximum over the empty set is $0$ per definitionem; that is, we have $\mu(a) = 0$ if and only if $n=0$ and $a = (a_0)$ is a sequence of length $1$.

\begin{theo}
\label{t:main}
Let $I,J \subset \R$ be intervals, $E$ a Banach space, $F = \End(E)$,
\[
G \colon I \to L^\infty_{loc}(J;F) \cap L^1(J;F)
\]
continuous with respect to the supremum distance on $F^J$ and of bounded variation with respect to $L^1(J;F)$.
Denote $V$ the total variation of $G$ with respect to $L^1(J;F)$ and put
\[
N := e^{\sup_{t \in I} \normx{G(t)}{L^1(J;F)}}.
\]
Then, for all $f \in AC_{loc}(I)$ such that $f(I) \subset J$, when
\[
A \colon I \to F, \quad A(t) = f'(t) \cdot G(t)(f(t)),
\]
we have
\begin{enumerate}
\item \label{i:main-l1} $A \in L^1_{loc}(I;F)$, and
\item \label{i:main-Xbound} if $X$ is an evolution operator for $A$ in $E$,
\begin{equation}
\label{e:Xbound}
\norm{X(t,s)^{\pm1}} \leq C := N^2 e^{N^{3 + 2N} V}
\end{equation}
for all $s,t \in I$ with $s \leq t$.
\end{enumerate}
\end{theo}

\begin{rema}
\label{r:main}
In case $I = \emptyset$, we can take the supremum in the definition of $N$ with respect to $\bar\R = [-\infty,\infty]$, in order to get $N = 0$, or with respect to $[0,\infty]$, in order to get $N = 1$; the theorem remains true both ways.

Assume there exists an element $t_0 \in I$. Then, for all $t \in I$, we have
\begin{align*}
\normx{G(t)}{L^1(J;F)} &\leq \normx{G(t) - G(t_0)}{L^1(J;F)} + \normx{G(t_0)}{L^1(J;F)} \\ &\leq V + \normx{G(t_0)}{L^1(J;F)}.
\end{align*}
Thus the supremum appearing in the definition of $N$ exists in $\R$.
\end{rema}

\begin{proof}[Proof of \cref{t:main}]
First of all, note that it suffices to prove the assertion in case $I$ is compact and nonempty. In order to deduce the general case from the special case, let $K \subset I$ be a compact, nonempty interval. Then applying the theorem to $K$, $G|_K$, and $f|_K$ in place of $I$, $G$, $f$, we deduce that $A|_K \in L^1_{loc}(K;F) = L^1(K;F)$. Furthermore, when $X$ is an evolution operator for $A$ in $E$, then $X|_{K\times K}$ is an evolution operator for $A|_K$. Thus \cref{e:Xbound} holds for all $s,t \in K$ with $s\le t$ since in passing from $K$, $G|_K$ to $I$, $G$, the numbers $V$, $N$, and thus $C$, can only get larger. As for all $s,t \in I$ with $s\le t$ there exists a compact nonempty interval $K \subset I$ such that $s,t \in K$, we infer \cref{i:main-Xbound}.

Assume $I$ compact and nonempty now. Let $f$ and $A$ be as above. Moreover, let $n \in \N$ and $a \in \mathcal S_n(\R)$ with $I = [a_0,a_n]$.\footnote{$a$ is what is called a \emph{partition} of $I$.} Define
\[
A_a \colon I \to F, \quad A_a(t) = f'(t)\cdot G(a_i)(f(t))
\]
whenever $i<n$ and $t \in \coint{a_i,a_{i+1}}$, or when $i = n$ and $t = a_n$. Then according to \cref{i:stability-l1} of \cref{l:stability}, we have $A_a \in L^1(I;F)$. According to \cref{i:stability-Xbound} of \cref{l:stability}, we have
\[
\norm{X_a(t,s)^{\pm1}} \leq C
\]
for all $s,t \in I$ with $s \leq t$ when $X_a$ is an evolution operator for $A_a$ in $E$; note that
\[
\max_{i<n} \normx{G(a_i)}{L^1(J;F)} \leq \sup_{t \in I} \normx{G(t)}{L^1(J;F)}
\]
and
\[
\sum_{0 < j < n} \normx{G(a_j) - G(a_{j-1})}{L^1(J;F)} \leq V.
\]
Moreover, for all $i < n$ and all $t \in \coint{a_i,a_{i+1}}$, but also for $i = n$ and $t = a_n$, we have
\begin{align*}
\norm{A(t) - A_a(t)} & = \abs{f'(t)}\cdot \normx{G(t)(f(t)) - G(a_i)(f(t))}F \\
& \leq \abs{f'(t)}\cdot d_\infty(G(a_i),G(t)) \\
& \leq \abs{f'(t)}\cdot \sup\set{d_\infty(G(a_i),G(\tau)) \mid \tau \in [a_i,a_{i+1}], i < n}.
\end{align*}

Observe that there exists a sequence $(a^k)_{k \in \N}$ of partitions of $I$ such that the corresponding sequence $(\mu(a^k))_{k \in \N}$ of meshes converges to $0$. I contend that the sequence $(A_{a^k})_{k \in \N}$ converges to $A$ pointwise on $I$. Indeed, as the function $G$ is continuous (with respect to the supremum distance) and $I \subset \R$ is compact, the function $G$ is uniformly continuous. In other words, for all numbers $\epsilon > 0$ there exists a number $\delta > 0$ such that
\[
d_\infty(G(\sigma),G(\tau)) < \epsilon
\]
whenever $\sigma,\tau \in I$ and $\abs{\tau - \sigma} < \delta$. Therefore, for all $\epsilon > 0$, there exists a number $k_0 \in \N$ such that for all $k \in \N$ with $k \geq k_0$ and all $t \in I$, we have
\[
\norm{A(t) - A_{a^k}(t)} \leq \abs{f'(t)}\cdot \epsilon.
\]
This proves the pointwise convergence.

Moreover (taking $\epsilon = 1$), we see that there exists a number $k_1 \in \N$ such that
\[
\norm{A(t) - A_{a^k}(t)} \leq \abs{f'(t)}
\]
holds for all $k \in \N$ with $k \geq k_1$ and all $t \in I$. In turn, for all $k \geq k_1$ and all $t \in I$,
\begin{align*}
\norm{A_{a^k}(t)} & \leq \norm{A_{a^k}(t) - A(t)} + \norm{A(t) - A_{a^{k_1}}(t)} + \norm{A_{a^{k_1}}(t)} \\
& \leq 2\abs{f'(t)} + \norm{A_{a^{k_1}}(t)}.
\end{align*}
Since $f' \in L^1(I;\R)$ (e.g., by \cite[Proposition 3.8]{MR2527916}) and $A_{a^l} \in L^1(I;F)$ for $l = 0,1,\dots,k_1$, this implies the existence of a function $g \in L^1(I;\R)$ such that $\norm{A_{a^k}} \le g$ for all $k \in \N$.
Hence by the dominated convergence theorem \cite[Theorem 1.1.8]{MR1886588}, $A$ is an element of $L^1(I;F)$ and
\[
\lim_{k \to \infty} \int_I \norm{A - A_{a^k}} \,d\lambda = 0.
\]
In particular, we have \cref{i:main-l1}.

Now let $X$ be an evolution operator for $A$ in $E$. Let $\epsilon > 0$. Then by the above there exists a natural number $l$ such that
\[
\int_I \norm{A - A_{a^l}} \,d\lambda < \epsilon.
\]
So by means of \cref{l:compare} we infer that, for all $s,t \in I$ such that $s \leq t$,
\[
\norm{X(t,s)^{\pm1}} \leq C e^{C \int_s^t \norm{A - A_{a^l}} \,d\lambda} < C e^{C \epsilon}.
\]
Yet as $\epsilon>0$ was arbitrary, this entails \cref{e:Xbound} for all $s,t \in I$ with $s\le t$; that is, we have established \cref{i:main-Xbound}.
\end{proof}

In order to demonstrate the strength of \cref{t:main} I would like to prove the bistability of some elementary, yet at first glance hard to tackle, systems.

\begin{coro}
\label{c:constant_G}
Let $I$, $E$, $F$ be as in \cref{t:main}, $\tilde G \in C(I;F)$ such that $\tilde G$ is of bounded variation with respect to $F$, and $f \in AC_{loc}(I)$ such that $f(I)$ is bounded in $\R$.
Let $A \colon I \to F$ be given by
\[
A(t) = f'(t)\tilde G(t), \quad \forall t \in I.
\]
Then $A$ is bistable in $E$ in the sense that there exists a number $C > 0$ such that, for all $s,t \in I$,
\begin{equation}
\label{e:constant_G}
\norm{X(t,s)} \le C
\end{equation}
when $X$ is an evolution operator for $A$ in $E$ (compare this notion of bistability to the one given in \cref{s:intro}).
\end{coro}

\begin{proof}
Observe that $J := f(I)$ is an interval. Let $G \colon I \to F^J$ be given by
\[
G(t)(u) = \tilde G(t)
\]
for all $t \in I$ and all $u \in J$. Then, for all $t \in I$, we have $G(t) \in L^\infty(J;F)$. Moreover, $L^\infty(J;F) \subset L^1(J;F)$ as $J$ is bounded. $G$ is evidently continuous with respect to the supremum distance on $F^J$, and
\[
V := V_{L^1(J;F)}(G) = \lambda(J)V_F(\tilde G) < \infty.
\]
Note that
\[
A(t) = f'(t)\cdot G(t)(f(t)), \quad \forall t\in I.
\]
Define
\[
N := e^{\sup_{t\in I}\normx{G(t)}{L^1(J;F)}} = e^{\lambda(J)\sup_{t\in I}\norm{\tilde G(t)}}
\]
and
\[
C := N^2e^{N^{3+2N}V}.
\]
Then $C>0$ and according to \cref{t:main}, \cref{e:constant_G} holds for all $s,t \in I$ when $X$ is an evolution operator for $A$ in $E$ (employ \cref{i:evolop-trans} of \cref{t:evolop} to write $X(t,s) = X(s,t)^{-1}$ in case $t< s$).
%So not only exists a bound $C$, but we can compute this bound effectively in terms of $E$, $\tilde G$, and the length of the interval $f(I)$.
%Besides, $A \in L^1_{loc}(I;F)$. The bound $C$ can be computed explicitly in terms of $\tilde G$, $J$, and $E$.
\end{proof}
%
%\begin{prop}
%\label{p:BV_components}
%Let $r\in \N$, $E = \R^r$ equipped with the Euclidean norm, $F = \End(E)$, and $\tilde G \colon I \to F$ a function such that $I \subset \R$. Then $\tilde G$ is of bounded variation with respect to $F$ if and only if $\tilde G_{ij}$ is of bounded variation with respect to $(\R,\abs\cdot)$ for all $(i,j) \in r\times r$, where $\tilde G_{ij}$ denotes the $(i,j)$ component of $\tilde G$ under the standard identification $\End(E) \cong \R^{r\times r}$.
%\end{prop}

\begin{exam}
\label{x:elementary}
Let $I = \coint{0,\infty}$, $E = \R^2$ (equipped with an arbitrary norm), $F = \End(E)$, and
\[
\tilde G(t) = \begin{pmatrix}2\arctan t&\sqrt{t+1}-\sqrt{t}\\-\frac1{1+t^2}&1 + e^{-t}\end{pmatrix}, \quad \forall t\in I,
\]
where we interpret the $2\times2$ matrices as elements of $F$ by virtue of the standard identification $\R^{2\times 2} \to \End(\R^2)$.

Then $I$ is an interval, $E$ a Banach space, and $\tilde G \in C(I;F)$. Moreover, each of the four component functions of $\tilde G$ is monotonic and bounded, whence of bounded variation with respect to the normed vector space $(\R,\abs\cdot)$. In consequence, $\tilde G$ is of bounded variation with respect to $F$. To establish the latter fact, note for one that the given norm on $F$ can be dominated by a constant times the $1$-norm on $F \cong \R^{2\times 2} \cong \R^4$ (as $F$ is finite-dimensional). For another, note that the total variation of $\tilde G$ with respect to the entrywise $1$-norm is bounded above by the sum of the total variations of the components $\tilde G_{ij}$, $(i,j) \in 2\times2$, with respect to $(\R,\abs\cdot)$.

By means of \cref{c:constant_G} we conclude that for all bounded functions $f \in AC_{loc}(I)$, the system
\[
A \colon I \to F, \quad A(t) = f'(t)\tilde G(t),
\]
is bistable in $E$. To make things entirely explicit, take $f(t) = \sin t$, for instance (observe that $C^1(I) \subset AC_{loc}(I)$). 

In the absence of \cref{t:main} it would be very much unclear how to decide whether $A$ is bistable in $E$ or not. Specifically the naive estimate of \cref{c:standard_estimate} turns out unfruitful. As $\tilde G(t)$ tends to $\tilde G(\infty) := \paren*{\begin{smallmatrix}\pi&0\\0&1\end{smallmatrix}}$ when $t$ tends to infinity and $\norm{\tilde G(\infty)} > 0$, there exists a number $\delta>0$ as well as an element $t_0 \in I$ such that
\[
\norm{A(\tau)} = \abs{f'(\tau)} \norm{\tilde G(\tau)} \geq \abs{\cos\tau} \delta, \quad \forall \tau > t_0.
\]
Therefore, for all $s \in I$, the integral $\int_s^t \norm{A} \,d\lambda$ tends to infinity as $t \in I$ tends to infinity. Compare this discussion to the one in \cref{s:intro}.
\end{exam}

\section{Bounds on parallel transports}
\label{s:bounds}

Let me recall how to compute the total variation of a nice (e.g., $C^1$) function.

\begin{prop}
\label{p:TV}
Let $a$ and $b$ be real numbers, $a\le b$, $F$ a Banach space, $g \in L^1([a,b];F)$, and $G \colon [a,b] \to F$ given by $G(t) = \int_a^t g \,d\lambda$ for all $t \in [a,b]$. Then
\[
V_F(G) = \int_{[a,b]} \norm{g} \,d\lambda.
\]
\end{prop}

\begin{proof}
See \cite[Proposition 1.2.2 d)]{MR1886588}.
\end{proof}

Let $a\le b$ be real numbers, $I = [a,b]$, and $F$ a Banach space. Then $C^1_{pw}(I;F)$ denotes the set of all \emph{piecewise $C^1$} functions from $I$ to $F$---that is, $G \in C^1_{pw}(I;F)$ if and only if $G \colon I \to F$ is a function such that there exist a natural number $m$ and an element $a \in \mathcal S_m(\R)$ such that $a_0 = a$, $a_m = b$, and
\[
G|_{[a_i,a_{i+1}]} \in C^1([a_i,a_{i+1}];F)
\]
for all $i<m$.

Let $G \in C^1_{pw}(I;F)$. Then $G'$ denotes the derivative of $G$, which is by definition the ordinary derivative of $G$ in points at which $G$ is differentiable and $0 \in F$ otherwise.

\begin{coro}
\label{c:TV-pwc1}
Let $a,b \in \R$, $a\le b$, $I = [a,b]$, $F$ a Banach space, $G \in C^1_{pw}(I;F)$. Then
\[
V_F(G) = \int_I \norm{G'} \,d\lambda.
\]
\end{coro}

\begin{proof}
First of all, we observe that $G' \in L^1(I;F)$. Second of all, we note that the fundamental theorem of calculus is valid for $G$ in the sense that, for all $t \in I$, we have
\[
G(t) - G(a) = \int_a^t G' \,d\lambda.
\]
Finally, as $V_F(G) = V_F(G - G(a))$, we are finished by virtue of \cref{p:TV}.
\end{proof}
%
%\begin{coro}
%\label{c:TV}
%Let $I \subset \R$ be an interval, $F$ a Banach space, $G \in C^1(I;F)$. Then
%\[
%V_F(G) = \int_I \norm{G'} \,d\lambda,
%\]
%where the integral might exist only in the improper sense.
%\end{coro}
%
%\begin{proof}
%Let $a,b \in I$ such that $a\le b$. Then $g := G'|_{[a,b]} \in L^1([a,b];F)$ as $G'$ is continuous. Define $\hat G \colon [a,b] \to F$ by $\hat G(t) = \int_a^t g \,d\lambda$ for all $t \in [a,b]$. Then $\hat G$ is differentiable on $[a,b]$ and satisfies $\hat G'(t) = g(t) = G'(t)$ for all $t \in [a,b]$. Thus $G|_{[a,b]}$ and $\hat G$ differ by a constant \cite[XIII, Lemma 3.3]{MR1216137}. In consequence,
%\[
%V_F(G|_{[a,b]}) = V_F(\hat G) = \int_{[a,b]} \norm{g} \,d\lambda = \int_a^b \norm{G'} \,d\lambda
%\]
%due to \cref{p:TV}.
%
%In case $I = \emptyset$, the overall assertion is clear. In case $I \neq \emptyset$, there exist sequences $(a_n)_{n\in\N}$ and $(b_n)_{n\in\N}$ of elements of $I$ such that $(a_n)$ is decreasing, $(b_n)$ is increasing, $a_0\le b_0$, and $I = \bigcup_{n \in \N} [a_n,b_n]$. We are finished noting that
%\[
%\lim_{n\to\infty} V_F(G|_{[a_n,b_n]}) = V_F(G)
%\]
%and
%\[
%\lim_{n\to\infty} \int_{a_n}^{b_n} \norm{G'} \,d\lambda = \int_I \norm{G'} \,d\lambda,
%\]
%where the latter is due to monotone convergence of course.
%\end{proof}

\begin{prop}
\label{p:TV-l1}
Let $a\le b$ be real numbers, $I = [a,b]$, $J$ an interval, $F$ a Banach space, $\tilde G \colon I \times J \to F$ a function such that there exist $m \in \N$ and $a \in \mathcal S_m(\R)$ such that $a_0 = a$, $a_m = b$, and, for all $i<m$, the restriction of $\tilde G$ to $[a_i,a_{i+1}] \times J$ is continuously partially differentiable with respect to the first variable. Assume that $G(t) := \tilde G(t,\_) \in L^1(J;F)$ for all $t \in I$. Then
\begin{equation}
\label{e:TV-l1}
V_{L^1(J;F)}(G) \le \int_{I \times J} \norm{D_1\tilde G} \,d\lambda^2,
\end{equation}
where the partial derivative $D_1\tilde G$ is understood to be zero in points at which $\tilde G$ is not partially differentiable with respect to the first variable.
\end{prop}

\begin{proof}
Let $n \in \N$ and $b \in \mathcal S_n(I)$. Then, for all $u \in J$,
\begin{align*}
\paren*{\sum_{k<n}\normx{G(b_{k+1}) - G(b_k)}F}(u)
&= \sum_{k<n} \normx{\tilde G(b_{k+1},u) - \tilde G(b_k,u)}F \\
&= \sum_{k<n} d_F(\tilde G(\_,u)(b_k),\tilde G(\_,u)(b_{k+1})) \\
&\le V_F(\tilde G(\_,u)) \\
&= \int_I \normx{D_1\tilde G(\_,u)}F \,d\lambda =: T(u),
\end{align*}
where we have employed \cref{c:TV-pwc1} in conjunction with the fact that $(\tilde G(\_,u))' = (D_1\tilde G)(\_,u)$ in the last line.

By our assumption on $\tilde G$, we know that the function $\norm{D_1\tilde G}$ is Lebesgue measurable on $I \times J$; in fact, for all $i<m$, the restriction of $\norm{D_1\tilde G}$ to $[a_i,a_{i+1}] \times J$ differs from a continuous function on a subset of $\{a_i,a_{i+1}\} \times J$ (i.e., on a null set). Thus due to Tonelli's theorem for nonnegative functions, the function $T$ is Lebesgue measurable on $J$; in fact, one can prove $T$ to be continuous. Furthermore,
\begin{align*}
\sum_{k<n} d_{L^1(J;F)}(G(b_k),G(b_{k+1})) &= \sum_{k<n} \normx{G(b_{k+1}) - G(b_k)}{L^1(J;F)} \\
&= \sum_{k<n} \int_J \normx{G(b_{k+1}) - G(b_k)}F \,d\lambda \\
&= \int_J \sum_{k<n}\normx{G(b_{k+1}) - G(b_k)}F \,d\lambda. \\
&\le \int_J T \,d\lambda = \int_{I\times J} \normx{D_1\tilde G}F \,d\lambda^2.
\end{align*}
The very last equality is again due Tonelli's theorem. As $n$ and $b$ were arbitrary, \cref{e:TV-l1} follows taking into account the definition of the total variation.
\end{proof}

Let $E$ and $F$ be normed vector spaces. Then $\Lin(E,F)$ denotes the normed vector space of continuous linear maps from $E$ to $F$. Moreover, $E \times F$ denotes the cartesian product in the sense of normed vector spaces where we use the (hyper) $1$-norm; that is,
$
\normx{(x,y)}{E\times F} = \normx{x}{E} + \normx{y}{F}
$
for all $x\in E$ and all $y \in F$. An equivalent norm would do equally fine.

Formally we deal with Banach manifolds and Banach bundles below. Since my considerations are of local nature, however, the general formalism \cite{MR1666820} might seem a bit excessive. A \emph{connection} is meant to be a linear connection (the latter in the sense of, e.g., Vilms \cite[236]{MR0229168}). When $E$ and $F$ are Banach spaces, $M$ is an open subset of $F$, viewed as a manifold, and $\bbE$ is the trivial Banach bundle with fiber $E$ over $M$ (i.e., $\bbE = M\times E \to M$), a connection on $\bbE$ corresponds univocally to a map
\[
\omega \colon M \to \Lin(F,\Lin(E,E)).
\]
Note that Vilms \cite{MR0229168} writes $\omega$ in a conjugated form---namely, as a map from $M \times E$ to $\Lin(F,E)$---since he must too account for nonlinear connections.

\begin{theo}
\label{t:bounds}
Let $E$ and $F$ be Banach spaces, $M \subset F$ open, $J \subset \R$ a bounded open interval,
\[
\omega \colon M \times J \to \Lin(F\times\R,\End(E))
\]
a bounded $C^1$ function such that $D_1\omega_2$ is bounded when $\omega_2$ denotes the second component of $\omega$ according to the natural decomposition
\[
\Lin(F\times\R,\End(E)) \cong \Lin(F,\End(E)) \oplus \Lin(\R,\End(E)).
\]
Let $\bbE$ be the trivial Banach bundle with fiber $E$ over $M\times J$ and $P$ the parallel transport in $\bbE$ with respect to the connection given by $\omega$.

Then there exists a monotonic $C^\infty$ function $\beta \colon \R \to (0,\infty)$ such that, for all real numbers $a\le b$ and all $\gamma \in C^1_{pw}([a,b];F\times \R)$ with $\gamma([a,b]) \subset M \times J$, we have
\begin{equation}
\label{e:bounds}
\normx{P_\gamma}{\Lin(\bbE_{\gamma(a)},\bbE_{\gamma(b)})} \le \beta(L(\gamma_1)),
\end{equation}
where $\gamma_1$ denotes the composition of $\gamma$ and the projection $\pi_1 \colon F \times \R \to F$ to the first factor and $L$ denotes the arc length of paths in $F$.\footnote{Note that $L$ is the same thing as the total variation $V_F$, just with a different name.}
\end{theo}

\begin{proof}
By assumption there exist bounds $B_1,B_2,B_{12}\ge 0$ for $\omega_1$, $\omega_2$, and $D_1\omega_2$ respectively. Note that $J$ is bounded, whence $0 \le \lambda(J) < \infty$. Set $N := e^{\lambda(J)B_2}$ and define the function $C$ by
\[
C(t) = N^2e^{N^{3+2N}\lambda(J)B_{12}t}, \quad \forall t\in\R.
\]
Define the function $\beta$ by
\[
\beta(t) = C(t)e^{C(t)B_1t}, \quad \forall t\in \R.
\]
Then evidently, $\beta \colon \R \to (0,\infty)$ is a monotonic function of class $C^\infty$.

Let $a\le b$ be real numbers, $I := [a,b]$, and $\gamma \colon I \to F\times \R$ a piecewise $C^1$ path whose image lies in $M \times J$. Put $\gamma_1 := \pi_1 \circ \gamma$, and define a function $\tilde G$ on $I \times J$ by
\[
\tilde G(t,u) = -\epsilon(\omega_2(\gamma_1(t),u))
\]
for all $t \in I$ and all $u \in J$, where
\[
\epsilon \colon \Lin(\R,\End(E)) \to \End(E), \quad \epsilon(\psi) = \psi(1),
\]
denotes the evaluation at the real number $1$.
Define the function $G$ on $I$ so that $G(t) = \tilde G(t,\_)$ for all $t \in I$. Then $G(t)$ is a continuous and bounded function from $J$ to $\End(E)$ for all $t \in I$. In particular,
\[
G \colon I \to L^\infty(J;\End(E)) \subset L^\infty_{loc}(J;\End(E)) \cap L^1(J;\End(E)),
\]
where we use that $J$ is bounded.

I contend that $G$ is continuous with respect to the supremum distance $d_\infty$ on $\End(E)^J$; for the definition of $d_\infty$ see the discussion before \cref{t:main}. As a matter of fact, as $\gamma_1 \colon I \to F$ is a piecewise $C^1$ path, we know that $D\gamma_1$ is bounded (in norm), say by $C_1 \in \R$. Let $u \in J$. Then the chain rule implies that, for all $\tau \in I$ (except possibly a finite number of points),
\[
(D_1\tilde G)(\tau,u) = - \epsilon \circ (D_1\omega_2)(\gamma_1(\tau),u) \circ (D\gamma_1)(\tau).
\]
Thus
\begin{equation}
\label{e:bounds-1}
\norm{(D_1\tilde G)(\tau,u)} \le B_{12} \norm{(D\gamma_1)(\tau)} \le B_{12}C_1,
\end{equation}
for $\norm{\epsilon} \le 1$. Hence, we have
\[
\norm{\tilde G(t,u) - \tilde G(s,u)} \le \abs{t - s} B_{12}C_1
\]
for all $s,t \in I$ as a consequence of the mean value theorem \cite[XIII, Corollary 4.3]{MR1216137}. As $u \in J$ was arbitrary, we deduce that
\[
d_\infty(G(s),G(t)) = \sup\set{d_{\End(E)}(G(s)(u),G(t)(u)) \mid u \in J} \le \abs{t - s} B_{12}C_1
\]
for all $s,t \in I$. In other words, $G \colon I \to \End(E)^J$ is Lipschitz continuous with respect to the supremum distance on $\End(E)^J$. The ordinary continuity of $G$ follows.

Using the first estimate in \cref{e:bounds-1} again, we obtain
\[
\int_I \norm{(D_1\tilde G)(\_,u)} \,d\lambda \le B_{12} \int_I \norm{\gamma_1'} \,d\lambda = B_{12}L(\gamma_1)
\]
for all $u \in J$. Thus by means of \cref{p:TV-l1} as well as Tonelli's theorem for nonnegative functions,
\[
V_{L^1(J;\End(E))}(G) \le \int_{I\times J} \norm{D_1\tilde G} \,d\lambda^2 \le \lambda(J)B_{12}L(\gamma_1).
\]
In particular we see that $G$ is of bounded variation with respect to $L^1(I;\End(E))$.

Denote $\gamma_2$ the composition of $\gamma$ and the projection $F\times \R \to \R$ to the second factor. Define the function $A_2$ on $I$ by
\[
A_2(t) = \gamma_2'(t)\cdot G(t)(\gamma_2(t)), \quad \forall t\in I.
\]
Then according to \cref{t:main}, $A_2 \in L^1(I;\End(E))$. Moreover, when $X_2$ is an evolution operator for $A_2$ in $E$, the estimate
\[
\norm{X_2(t,s)} \le C(L(\gamma_1)) =: C
\]
holds for all $s,t\in I$ with $s\le t$. Observe here that, for all $t \in I$ and all $u \in J$,
\[
\norm{G(t)(u)} = \norm{\tilde G(t,u)} \le \norm{\omega_2(\gamma_1(t),u)} \le B_2,
\]
so that
\[
\normx{G(t)}{L^1(J;\End(E))} = \int_J \norm{G(t)} \,d\lambda \le \lambda(J)B_2
\]
holds for all $t \in I$.

Let the function $A$ on $I$ be given by
\[
A(t) = - (\omega(\gamma(t)) \circ (D\gamma)(t))(1)
\]
for all $t \in I$. Then
\[
A(t) = - (\omega_1(\gamma(t))\circ (D\gamma_1)(t) + \omega_2(\gamma(t))\circ (D\gamma_2)(t))(1),
\]
whence
\[
(A - A_2)(t) = - (\omega_1(\gamma(t))\circ (D\gamma_1)(t))(1)
\]
as
\[
A_2(t) = -\gamma_2'(t)\cdot (\omega_2(\gamma_1(t),\gamma_2(t)))(1) = -\paren*{\omega_2(\gamma(t)) \circ (D\gamma_2)(t)}(1)
\]
for all $t \in I$. Thus
\[
\int_s^t \norm{A - A_2} \,d\lambda \le \int_s^t B_1 \norm{D\gamma_1} \,d\lambda = B_1 \int_s^t \norm{\gamma_1'} \,d\lambda \le B_1L(\gamma_1)
\]
for all $s,t\in I$ with $s\le t$. Let $X$ be an evolution operator for $A$ in $E$. Then according to \cref{l:compare}, we have
\[
\norm{X(t,s)} \le Ce^{C \int_s^t \norm{A - A_2} \,d\lambda} \le \beta(L(\gamma_1))
\]
for all $s,t \in I$ such that $s\le t$.

By the very definition of the parallel transport in $\bbE$ with respect to the connection given by $\omega$, the function $P_\gamma \colon \bbE_{\gamma(a)} \to \bbE_{\gamma(b)}$ corresponds to $X(b,a) \colon E \to E$, plugging in the Banach space isomorphisms $\bbE_{\gamma(a)} \to E$ and $\bbE_{\gamma(b)} \to E$ which are given by the projection $\bbE = (M\times J)\times E \to E$ to the second factor. Therefore, we deduce \cref{e:bounds}, which was to be demonstrated.
\end{proof}

\begin{exam}
\label{x:topologists_sin}
Let $M,J \subset \R$ be open intervals. Assume that $0 \in M$ and $[-1,1] \subset J$. Let $r$ be a natural number, $\bbE$ the trivial rank-$r$ bundle over $M \times J$, $\nabla$ a smooth (i.e., $C^\infty$) connection on $\bbE$, $a \in M$ with $a<0$, and $\gamma \colon \coint{a,0} \to \R \times \R$ so that
\[
\gamma(t) = \paren*{t,\sin\frac1t}, \quad \forall t \in \coint{a,0}.
\]
I contend there exists a real number $C > 0$ such that, for all $b \in \coint{a,0}$ and all $v \in \bbE_{\gamma(a)}$, we have
\begin{equation}
\label{e:topologists_sin}
\frac1C \norm{v} \le \norm{P_{\gamma_b}(v)} \le C\norm{v}
\end{equation}
where $P$ signifies the parallel transport in $\bbE$ with respect to $\nabla$, and $\gamma_b = \gamma|_{[a,b]}$.

Indeed, \cref{t:bounds} implies the existence of a monotonic function $\beta \colon \R \to (0,\infty)$ such that, for all real numbers $c\le d$ and all piecewise $C^1$ functions $\delta \colon [c,d] \to \R \times \R$ with $\delta([c,d]) \subset [a,b] \times [-1,1]$, one has
\[
\norm{P_\delta} \le \beta(L(\delta_1))
\]
where $\delta_1$ denotes the first component of $\delta$ and $L$ measures the arc length of paths in $\R$. Specifically, we obtain, for all $b \in \coint{a,0}$,
\[
\norm{P_{\gamma_b}} \le \beta(L(\gamma_1)) = \beta(b-a) \le \beta(-a) =: C.
\]
Hence the upper bound in \cref{e:topologists_sin}. The lower bound is obtained looking at inverse path $\gamma_b^{-1}$ of $\gamma_b$ instead of $\gamma_b$, noting that $P_{\gamma_b^{-1}}\circ P_{\gamma_b} = \id_{\bbE_{\gamma(a)}}$.
\end{exam}

\section{Negligible function graphs}
\label{s:neg}

In what follows, a \emph{manifold} is a real differentiable manifold of class $C^\infty$ locally modeled on $\R^n$ for a natural number $n$. A \emph{vector bundle} is understood the same way; it is assumed to be real (as opposed to complex). We do not deal with Banach manifolds and Banach bundles here.

A \emph{connection} on a vector bundle is, still, a continuous linear connection. For $m \in \N$, or $m = \infty$, we say that a connection is \emph{of class $C^m$} when all of its local components \cite{MR0229168} are of class $C^m$. When $M$ is a manifold, $\bbE$ a vector bundle over $M$, and $\nabla$ a connection on $\bbE$, a \emph{$\nabla$-parallel section} in $\bbE$ is a $C^1$ section $\sigma$ in $\bbE$ defined on an open subset $U$ of $M$ such that, for all $p \in U$ and all tangent vectors $e \in T_p(M)$, the covariant derivative of $\sigma$ in the direction of $e$ vanishes. The $\nabla$-parallel sections in $\bbE$ naturally form a subsheaf of the sheaf of $C^1$ sections in $\bbE$.

The following definition stems from a previous paper of mine \cite{extendability}.

\begin{defi}
\label{d:neg}
Let $M$ be a manifold, $F$ a closed subset of $M$, $m \in \N \cup\set\infty$. Then $F$ is called \emph{negligible} in $M$ for all connections of class $C^m$ when, for all vector bundles $\bbE$ over $M$ and all connections $\nabla$ of class $C^m$ on $\bbE$, the restriction map
\[
\rho_{M,M\setminus F} \colon H(M) \to H(M \setminus F)
\]
for the sheaf $H$ of $\nabla$-parallel sections in $\bbE$ is surjective.
\end{defi}

In case $M$ is a, say simply connected\footnote{This assumption can be somewhat weakened.}, second-countable Hausdorff manifold of dimension $\ge2$ and $F$ is a closed submanifold, boundary allowed, of class $C^1$ of $M$, we know \cite{extendability} that $F$ is negligible in $M$ for all connections of class $C^0$ if and only if $M \setminus F$ is dense and connected in $M$. This result relies heavily on the fact that when $F \subset M$ is a closed $C^1$ submanifold, with possible boundary, $F$ can be locally flattened by means of a diffeomorphism. Already when $F$ is only a $C^0$ submanifold of $M$ (an only in the boundary points of $F$), the suggested method of proof fails.

In one of his talks, Antonio J. Di Scala asked whether the closed topologist's sine curve---that is, the closure of the graph of $f(t) = \sin\frac1t$, $t>0$---was negligible in $\R^2$ for all connections of class $C^\infty$. As an application of \cref{t:main}, I prove that this is indeed true (observe the connection with \cref{x:topologists_sin}). More generally, the following holds.

\begin{theo}
\label{t:neg}
Let $M,J \subset \R$ be open intervals, $a \in M$, $M_{>a} = \{x \in M \mid x>a\}$, $f \in C(M_{>a})$ such that $f(M_{>a})$ is relatively compact in $J$---that is, the closure in $J$ of $f(M_{>a})$ is compact.
Then the closure in $M\times J$ of the graph of $f$ is negligible in $M\times J$, thought of as a manifold, for all connections of class $C^1$.
\end{theo}

An indispensable tool in the proof of \cref{t:neg} is the solving of parameter-dependent linear evolution equations.

\begin{theo}
\label{t:parameters}
Let $J \subset \R$ be an open interval, $v_0 \in J$, $E$ and $F$ Banach spaces, $M$ an open subset of $F$, $A \colon M \times J \to \End(E)$ a continuous map.
\begin{enumerate}
\item \label{i:parameters-ex} There exists a unique continuous map $X \colon M \times J \to \End(E)$ which is partially differentiable with respect to the second variable such that
\[
(D_2X)(p) = A(p) \circ X(p), \quad \forall p \in M\times J,
\]
and
\[
X(x,v_0) = \id_E, \quad \forall x \in M.
\]
\item \label{i:parameters-c1} When $A$ is of class $C^1$, then the $X$ from \cref{i:parameters-ex} is of class $C^1$.
\end{enumerate}
\end{theo}

\begin{proof}
See \cite[Theorem 3.1]{extendability} where I comment on Lang's exposition \cite[IV, \S1]{MR1666820}.
\end{proof}

The following approximation lemma permits us to look at arbitrary continuous functions, as opposed to only $C^1$ functions, $f$ in \cref{t:neg}.

\begin{lemm}
\label{l:approximation}
Let $I \subset \R$ be a compact interval, $b \in I$, $f \in C(I)$, $U \subset \R^2$ open such that $U$ contains the graph of $f$.\footnote{Note that in ZF set theory, $f$, as a set, \emph{is} the graph of $f$.} Then there exists $g \in C^\infty(I)$ such that $g(b) = f(b)$ and the graph of $g$ is contained in $U$.
\end{lemm}

\begin{proof}
First of all, observe that, for all $t \in I$, there exist numbers $\delta,\epsilon > 0$ such that
\[
f|_{B_\delta(t)} \subset B_\delta(t) \times B_\epsilon(f(t)) \quad \text{and} \quad B_\delta(t) \times B_{2\epsilon}(f(t)) \subset U,
\]
where $B$ denotes open Euclidean balls in $\R$. In particular, the sets $B_\delta(t)$ associated to such “distinguished” triples $(t,\delta,\epsilon)$ furnish an open cover of $I$ in $\R$. As $I$ is compact in $\R$, there exists a natural number $m$ and an $m$-tuple $((t_i,\delta_i,\epsilon_i))_{i<m}$ of distinguished triples such that
\[
I \subset \bigcup_{i<m} B_{\delta_i}(t_i).
\]

Since $b \in I$, we have $m>0$. Set $\epsilon := \min_{i<m}(\epsilon_i)$. Let $(s,v) \in I \times \R$ such that $\abs{v - f(s)} < \epsilon$. Then there exists an index $i<m$ such that $s \in B_{\delta_i}(t_i)$. It is implied that
\[
\abs{v - f(t_i)} \leq \abs{v - f(s)} + \abs{f(s) - f(t_i)} < \epsilon + \epsilon_i \leq 2\epsilon_i,
\]
whence $(s,v) \in U$.
%We conclude
%\[
%\{(t,v) \in I \times \R \mid \abs{v - f(t)} < \epsilon\} \subset U.
%\]

By the Weierstraß approximation theorem, noting $\epsilon > 0$, there exists a polynomial function $p$ on $I$ such that, for all $s \in I$, we have $\abs{p(s) - f(s)} < \frac\epsilon2$. Define
\[
g \colon I \to \R, \quad g = (f(b) - p(b)) + p.
\]
Then $g \in C^\infty(I)$. Moreover, $g(b) = f(b)$ and, for all $s \in I$,
\[
\abs{g(s) - f(s)} \leq \abs{f(b) - p(b)} + \abs{p(s) - f(s)} < \frac\epsilon2 + \frac\epsilon2 = \epsilon,
\]
so that $(s,g(s)) \in U$, which was to be demonstrated.
\end{proof}

\begin{proof}[Proof of \cref{t:neg}]
Let $\bbE$ be a vector bundle over $M\times J$, $\nabla$ a connection of class $C^1$ on $\bbE$. Then, as $M\times J$ is paracompact, Hausdorff, and $C^\infty$ contractible, the vector bundle $\bbE$ is trivial---that is, there exists a number $r \in \N$ and a vector bundle isomorphism $\psi \colon \bbE \to (M\times J) \times \R^r$. Let
\[
\omega \colon M \times J \to \Lin(\R \times \R,\End(\R^r))
\]
be the local component of $\nabla$ with respect to $\psi$. Denote $\omega_1$ and $\omega_2$ the components of $\omega$ with respect to the natural decomposition
\[
\Lin(\R \times \R,\End(\R^r)) \cong \Lin(\R,\End(\R^r)) \oplus \Lin(\R,\End(\R^r)) \cong \End(\R^r) \oplus \End(\R^r),
\]
where the second isomorphism is given by the evaluation at $1 \in \R$.

Let $\tilde\sigma \in H(U)$ where $H$ is the sheaf of $\nabla$-parallel sections in $\bbE$, $F$ is the closure in $M \times J$ of the graph of $f$, and $U := (M\times J)\setminus F$. Denote $\sigma$ the composition of $\tilde\sigma$, $\psi$, and the projection to the second factor (i.e., to $\R^r$). Then the fact that $\nabla(\tilde\sigma) = 0$ in the covariant derivative sense implies
\[
D\sigma + \omega\sigma = 0,
\]
as functions from $U$ to $\Lin(\R\times\R,\R^r)$.
Passing to the components and evaluating at $1 \in \R$ as above, we obtain, for $i = 1,2$,
\[
D_i\sigma + \omega_i\sigma = 0,
\]
as functions from $U$ to $\R^r$. Here $(\omega_i\sigma)(p) = \omega_i(p)(\sigma(p))$ for all $p \in U$ and $i = 1,2$.

Since $f(M_{>a})$ is relatively compact in $J$, there exist elements $v_0<v_1$ in $J$ such that $\bar{f(M_{>a})}$ is contained in the open interval between $v_0$ and $v_1$. In particular,
\[
F \subset \{x \in M \mid x\ge a\} \times \bar{f(M_{>a})} \subset \{x \in M \mid x\ge a\} \times (v_0,v_1).
\]
Let $j\in\set{0,1}$. Then by \cref{t:parameters}, as $\omega_2$ is of class $C^1$, there exists a (unique) $C^1$ function
\[
Y_j \colon M \times J \to \End(\R^r)
\]
such that
\[
D_2Y_j + \omega_2Y_j = 0
\]
holds on $M\times J$
and $Y_j(\_,v_j)$ is constantly equal to $\id_{\R^r}$. Since $M\times\set{v_j} \subset U$, there exists a function $\xi_j$ on $M\times J$ such that
\[
\xi_j(x,v) = Y_j(x,v)(\sigma(x,v_j)), \quad \forall (x,v) \in M\times J.
\]
Clearly, $\xi_j \in C^1(M\times J;\R^r)$.
In particular, $\xi_j$ is partially differentiable with respect to the first variable. Put
\[
\theta_j := D_1\xi_j + \omega_1\xi_j.
\]

Define
\[
U_{>a}^j := \{(x,v) \in M_{>a} \times J \mid 0 < (-1)^j(f(x) - v)\}.
\]
Thus $U_{>a}^0$ is the portion of $M\times J$ which lies beneath the graph of $f$, whereas $U_{>a}^1$ is the portion of $M\times J$ which lies above the graph of $f$.
Observe that $\xi_j = \sigma$ holds on $U_{>a}^j$. As a matter of fact, let $x \in M_{>a}$. Then both $\xi_j(x,\_)$ and $\sigma(x,\_)$ are vector solutions for $-\omega_2(x,\_)$ in $\R^r$ when restricted to the interval
\[
\{v \in J \mid 0 < (-1)^j(f(x) - v)\}.
\]
Besides, $\xi_j(x,v_j) = \sigma(x,v_j)$. Thus $\xi_j(x,\_)$ and $\sigma(x,\_)$ agree on the latter interval, for the initial value problem has a unique solution. By the exact same argument, one infers that $\xi_j = \sigma$ on $M_{<a} \times J$ where $M_{<a} := \{x \in M \mid x < a\}$.

In consequence, as $U_{>a}^j$ is open in $M\times J$, or equivalently in $\R\times \R$, we infer that on $U_{>a}^j$,
\[
\theta_j = D_1\sigma + \omega_1\sigma = 0.
\]
The exact same equation holds on $M_{<a} \times J$. As $\theta_j$ is continuous with respect to the second variable, we deduce that $\theta_j(x,f(x)) = 0$ for all $x \in M_{>a}$. As $\theta_j$ is continuous with respect to the first variable, we deduce that $\theta_j(a,v) = 0$ for all $v \in J$.

Let $b \in M_{>a}$. I contend that
\[
\xi_0(b,f(b)) = \xi_1(b,f(b)).
\]
According to \cref{t:bounds}, there exists a monotonic function $\beta\colon\R \to (0,\infty)$ such that, for all real numbers $c\le d$ and all $\delta \in C^1_{pw}([c,d];\R\times\R)$ with $\delta([c,d]) \subset [a,b] \times [v_0,v_1]$, we have
\[
\normx{P_\delta}{\Lin(\R^r,\R^r)} \le \beta(L(\delta))
\]
when $P$ is the parallel transport in $(M\times J)\times \R^r$ with respect to $\omega$, and $\delta_1$ denotes the composition of $\delta$ with the projection to the first factor, and $L$ stands for the arc length of paths in $\R$.

Let $\epsilon > 0$. Put $\theta := \theta_1 - \theta_0$ and
\[
F_\epsilon := \{p \in M\times J \mid \norm{\theta(p)} < \epsilon\}.
\]
Then $F_\epsilon$ is open in $M\times J$, and $F_\epsilon$ contains the graph of $f$ as well as the set $\set{a} \times J$.
Thus employing \cref{l:approximation}, we see there exists a function $h \in C^1_{pw}([a,b])$ such that the graph of $h$ lies in $F_\epsilon \cap (M \times (v_0,v_1))$ and $h(b) = f(b)$. Define the function $\gamma$ on $[a,b]$ by
\[
\gamma(t) = (t,h(t)), \quad \forall t\in [a,b].
\]
Then, for all $t \in [a,b]$, by virtue of the chain rule,
\begin{align*}
(\xi_j\circ \gamma)'(t)
& = (D_1\xi_j)(\gamma(t))\gamma_1'(t) + (D_2\xi_j)(\gamma(t))\gamma_2'(t) \\
& = \theta_j(\gamma(t)) - \omega(\gamma(t))(\gamma'(t))(\xi_j(\gamma(t))).
\end{align*}
Hence,
\[
\Delta'(t) = \theta(\gamma(t)) + A(t)\Delta(t), \quad \forall t \in [a,b],
\]
where
\[
\Delta := (\xi_1 - \xi_0) \circ \gamma = (\xi_1\circ\gamma) - (\xi_0\circ\gamma)
\]
and the function $A$ on $[a,b]$ is given by
\[
A(t) = -\omega(\gamma(t))(\gamma'(t)), \quad \forall t \in [a,b].
\]

By \cref{t:evolop}, there exists an evolution operator $X$ for $A$ in $\R^r$. For all $s,t \in [a,b]$ such that $s\le t$, the operator $X(t,s)$ is by definition the parallel transport along $\gamma|_{[s,t]}$ in $(M\times J)\times \R^r$ with respect to $\omega$. Therefore,
\[
\norm{X(t,s)} = \norm{P_{\gamma|_{[s,t]}}} \le \beta(L(\gamma_1|_{[s,t]})) = \beta(t-s) \le \beta(b-a).
\]
In addition, the variation of parameters formula \cite[101]{MR0352639} yields
\[
\Delta(b) = X(b,s)\Delta(s) + \int_s^b X(b,\_)(\theta\circ\gamma) \,d\lambda
\]
for all $s \in [a,b]$.
Specifically, since
\[
\Delta(a) = \xi_1(a,h(a)) - \xi_0(a,h(a)) = \sigma(a,h(a)) - \sigma(a,h(a)) = 0,
\]
we obtain
\[
\norm{\Delta(b)} = \norm*{\int_a^b X_2(b,\_)(\theta\circ\gamma) \,d\lambda} \leq \int_a^b \beta(b-a)\epsilon \,d\lambda \leq (b-a)\beta(b-a)\epsilon.
\]
Since $\epsilon > 0$ was arbitrary, we deduce $\Delta(b) = 0$, whence $\xi_0(b,f(b)) = \xi_1(b,f(b))$, as claimed.

Since both $\xi_0(b,\_)$ and $\xi_1(b,\_)$ are vector solutions for $-\omega_2(b,\_)$ in $\R^r$, the fact that they agree at one point---namely, at $f(b)$---implies that they agree as such (i.e., as functions).
As moreover $b \in M_{>a}$ was arbitrary, we conclude that $\xi_0 = \xi_1$ holds on all of $M_{>a} \times J$. In turn, $\xi_0|_U = \sigma$. Evidently, there exists a $C^1$ section $\tilde\xi_0 \colon M\times J \to \bbE$ in $\bbE$ such that the composition of $\tilde\xi_0$, $\psi$, and the projection to the second factor (i.e., to $\R^r$) equals $\xi_0$. By construction, we have $\tilde\xi_0|_U = \tilde\sigma$ as well as $\nabla(\tilde\xi_0) = 0$ in the covariant derivative sense; that is, $\tilde\xi_0 \in H(M\times J)$. As $\tilde\sigma \in H(U)$ was arbitrary, this proves the surjectivity of the restriction map
\[
\rho_{M\times J,U} \colon H(M\times J) \to H(U)
\]
for the sheaf $H$. As $\bbE$ and $\nabla$ were arbitrary, we deduce further that $F$ is negligible in $M\times J$ for all connections of class $C^1$, which was to be demonstrated.
\end{proof}

\printbibliography
\end{document}